\documentclass[a4paper]{article}

\usepackage{geometry}

\usepackage{palatino}
\usepackage{latexsym}

\usepackage[mathscr]{eucal}
\usepackage{amsmath}
\usepackage{amsthm}
\usepackage{amsfonts}
\usepackage{amssymb}
\usepackage{amscd}
\usepackage{color}
\usepackage{graphicx}
\usepackage{graphics}
\usepackage{pifont}
\usepackage{subfigure}
\usepackage[makeroom]{cancel}
\usepackage[normalem]{ulem}
\usepackage[dvipsnames]{xcolor}
\usepackage{tikz}

\theoremstyle{plain}
\newtheorem{theorem}{Theorem}
\newtheorem{remark}[theorem]{Remark}
\newtheorem{lemma}[theorem]{Lemma}
\newtheorem{proposition}[theorem]{Proposition}
\newtheorem{definition}[theorem]{Definition}

\usepackage{newtxtext}       %
\usepackage{newtxmath}  
\usepackage{tikz}

\newcommand{\ii}{\mathrm{i}}
\newcommand{\cb}{{\mathbf U}}
\newcommand{\fibra}{U}
\newcommand{\metrup}[3]{( #1\, |\,#2 )_{#3}}
\newcommand{\distr}{\mathcal{D}}
\newcommand{\norm}[1]{\| #1 \|}
\newcommand{\ff}{{\mathfrak f}}
\title{Quantum confinement for the curvature Laplacian $-\Delta+cK$ on 2D-almost-Riemannian manifolds}

\begin{document}
\author{Ivan Beschastnyi\footnote{CIDMA - Centro de I\&D em Matematica e Aplicações, Departamento de 
Matematica Campus Universitario de Santiago, 3810-193, Aveiro, Portugal (i.beschastnyi@gmail.com)}  \and 
Ugo Boscain\footnote{CNRS, Sorbonne Universit\'{e}, Inria, Universit\'{e} de Paris, Laboratoire Jacques-Louis Lions, Paris, France.
              (ugo.boscain@upmc.fr)  }        \and
         Eugenio~Pozzoli\footnote{Inria, Sorbonne Universit\'e, Universit\'e de Paris, CNRS, Laboratoire Jacques-Louis Lions, Paris, France (eugenio.pozzoli@inria.fr).}}
         
         \maketitle
\begin{abstract}
Two-dimension almost-Riemannian structures of step 2 are natural generalizations of the Grushin plane.
They are generalized Riemannian structures for which the vectors of a local orthonormal frame can become parallel.
Under the 2-step assumption the singular set $Z$, where the structure is not Riemannian, is a 1D embedded 
submanifold. While approaching the singular set, all Riemannian quantities diverge. A remarkable property of 
these structures is that the geodesics can cross the singular set without singularities, but the heat and the solution of the 
Schr\"{o}dinger equation (with the Laplace-Beltrami operator $\Delta$) cannot. This is due to the fact that (under a natural 
compactness hypothesis), the Laplace-Beltrami operator is  essentially self-adjoint on a connected component of the 
manifold without the singular set. In the literature such phenomenon is called quantum confinement.

In this paper we study the self-adjointness of the curvature Laplacian, namely 
$-\Delta+cK$, for $c\in(0,1/2)$ (here $K$ is the Gaussian curvature), which originates 
in coordinate-free quantization procedures (as for instance in path-integral or covariant Weyl quantization). 
We prove that there is no quantum confinement for this type of operators.
\end{abstract}
\section{Introduction}

A $2$-dimensional  almost-Riemannian Structure ($2$-ARS for short) is a generalized Riemannian structure on a $2$-dimensional manifold $M$, that can be defined locally by assigning a pair of smooth vector fields, which play the role of an orthonormal frame. 
It is assumed that the vector fields satisfy the H\"ormander condition (see Section \ref{s-ARS} for a more intrinsic definition). 

$2$-ARSs were introduced in the context of hypoelliptic operators \cite{FL1,grushin1} and are particular case of rank-varying sub-Riemannian structures (see for instance \cite{ABB,bellaiche,jean1,Montgomery-Subriemannian-2002,jean2}). The geometry of $2$-ARSs was studied in \cite{ABS,euler,kresta,lip} while several questions of geometric analysis on such structures were analyzed in \cite{ivan,Boscain-Laurent-2013,seri,Franceschi-Prandi-Rizzi-2017,GMP-Grushin-2018,PozzoliGru-2020volume,Prandi-Rizzi-Seri-2016}.
 For an easy introduction see \cite[Chapter 9]{ABB}.
$2$-ARSs appear also in applications; for instance  in  \cite{q4,q1} for  problems of  population transfer in quantum systems  and in \cite{BCa,tannaka} for  orbital transfer in space mechanics.

Let us denote by $\mathcal{D}_p$  the linear span of the two vector fields at a  point $p$.  Where  $\mathcal{D}_p$ is $2$-dimensional, the corresponding metric is Riemannian. On the singular set $Z$, where $\mathcal{D}_p$ is $1$-dimensional, the corresponding Riemannian metric is not well defined, but thanks to the H\"ormander condition one can still define the Carnot-Carath\'eodory distance between two points, which happens to be finite and continuous. The H\"ormander condition prevents the existence of points where $\mathcal{D}_p$ is zero dimensional. When the set $Z$ is non-empty, we say that the 2-ARS is {\em genuine}.

In most part of  the paper we make the hypothesis that the $2$-ARS is 2-step i.e., that for every $p\in M$ we have dim$(\mathcal{D}_p+[\mathcal{D},\mathcal{D}]_p)=2$. Such a hypothesis guarantees that  the singular set $Z$ is a (closed) $1$-dimensional embedded submanifold and that for every $p\in Z$  we have that $\mathcal{D}_p$ is transversal to $Z$.

One of the main features of $2$-ARSs is the fact that geodesics can pass through the singular set, with no singularities even if all Riemannian quantities (as for instance the metric, the Riemannian area, the curvature) explode while approaching $Z$. 

This is easily illustrated with the example of the Grushin cylinder that is the 2-ARS  on ${\bf R}\times S^1$ defined by the vector fields
$$
X_1(x,y)=\dfrac{\partial}{\partial x}, \qquad X_2(x,y)=x \dfrac{\partial}{\partial y},\mbox{ ~~~ here } x\in{\bf R},~~y\in S^1.
$$ 

\begin{figure}[ht!]\begin{center}
\includegraphics[width=0.4\linewidth, draft = false]{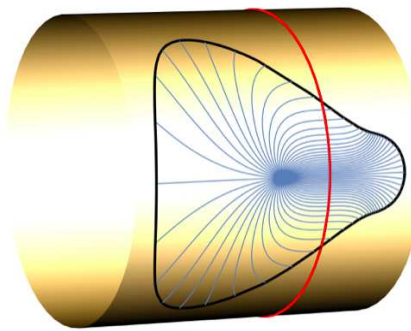}
\caption{Geodesics on the Grushin cylinder, staring from the point $(-1/2,0)$ with final time $t_f=1.3$, crossing smoothly the singular set $Z$ (red circle).} \label{grushin}
\end{center}\end{figure}

For such a structure, the geodesics cross the singular set $Z=\{(x,y)\in{\bf R} \times S^1\mid x=0\}$ without singularities (see Figure \ref{grushin}),
while the Riemannian metric $g$, the Riemannian area $\omega$ and the Gaussian curvature $K$ are deeply singular on $Z$:
\begin{align}
g
=\left(\begin{array}{cc}1&0\\0&1/x^2 \end{array}\right),~~~\omega
=\frac{1}{|x|}dx\,dy,~~~K
=-\frac{2}{x^2}.
\end{align}

However even if geodesics cross the singular set, this is not possible for the Brownian motion or for a quantum particle when they are described by the Laplace-Beltrami operator $\Delta$ associated to the $2$-ARS. This is due to the explosion of the Riemannian area while approaching $Z$ that makes appearing  highly singular first order terms in $\Delta$.   For instance for the Grushin cylinder the Laplace-Beltrami operator is given by
$$
\Delta
={\rm div}_\omega\circ {\rm grad}_g=\dfrac{\partial^2}{\partial x^2}+x^2\dfrac{\partial^2}{\partial y^2}-\frac1x\dfrac{\partial}{\partial x}.
$$

This phenomenon is described by the following Theorem.\footnote{In \cite{Boscain-Laurent-2013} there is the additional hypothesis that $Z$ is en embedded one-dimensional submanifold of $M$. However, as a direct consequence of the implicit function theorem, it is easy to see that such an hypothesis is implied by the fact that the structure is $2$-step.}

\begin{theorem}[\cite{Boscain-Laurent-2013}]\label{prototheorem}
Let $M$ be a compact oriented $2$-dimensional manifold equipped with a genuine 2-step $2$-ARS.  Let $\Omega$ be a connected component of $M\setminus Z$, where $Z$ is the singular set.  Let $g$ be the Riemannian metric induced by the $2$-ARS on $\Omega$ and $\omega$ be the corresponding Riemannian area. The Laplace-Beltrami operator $\Delta:={\rm div}_\omega\circ {\rm grad}_g$ with domain $C^\infty_0(\Omega)$, is essentially self-adjoint on $L^2(\Omega, \omega)$.
\end{theorem}

Notice that by construction each connected component of $\partial\Omega$ is diffeomorphic to $S^1$, $\Omega$ is open and $(\Omega,g)$ is a non-complete Riemannian manifold. In Theorem \ref{prototheorem} the compactness hypothesis is not necessary, but simplifies the statement.
In particular the conclusion of the theorem holds for the Grushin cylinder. Versions of Theorem \ref{prototheorem} in more general settings have been proved in \cite{Franceschi-Prandi-Rizzi-2017,Prandi-Rizzi-Seri-2016}.

The main consequence of Theorem \ref{prototheorem} is that the Cauchy problems  for the heat and the Schr\"{o}dinger equations (here $\Delta$ is the Laplace-Beltrami operator, $\hbar$ is the Planck constant, and $m$ is the mass of the quantum particle)\footnote{For the heat equation all constant are normalized to 1. This has not been done for the Schr\"odinger equation since the role of the Planck constant $\hbar$ is important for further discussions.} 
\begin{eqnarray}
\partial_t\phi(t,p)&=&\Delta\,\phi(t,p),~~~~~~~~~~~~~~\phi(0,\cdot)=\phi_0\in L^2(\Omega,\omega),\\
\ii\hbar\,\partial_t\psi(t,p)&=&-\frac{\hbar^2}{2m}\Delta\,\psi(t,p),~~~~~~~\psi(0,\cdot)=\psi_0\in L^2(\Omega,\omega).
\end{eqnarray}
are well defined in $L^2(\Omega, \omega)$ and hence nothing can flow outside $\Omega$, that is, 
$e^{t\Delta}\phi_0$ (resp. $e^{\ii t\frac{\hbar}{2m}\Delta}\psi_0$)  is supported in $\Omega$, for all $t\geq0$ (resp. $t\in {\bf R}$).
This phenomenon is usually known as {\em quantum confinement} (see \cite{Franceschi-Prandi-Rizzi-2017,Prandi-Rizzi-Seri-2016} and, for similar problems, \cite{nenciu}).

Given that the geodesics cross the singular set with no singularities, the impossibility for the heat or for a quantum particle to flow through $Z$ implied by Theorem  \ref{prototheorem} is quite surprising.
For what concerns the heat, a satisfactory interpretation of Theorem \ref{prototheorem} in terms of Brownian motion/Bessel processes has been provided for the Grushin cylinder in \cite{boscain-neel} and from \cite{volume-sampling} one can extract an interpretation of Theorem \ref{prototheorem} in terms of  random walks. Roughly speaking random particles are lost in the infinite area accumulated along $Z$ that, as a consequence, acts as a barrier.

Although for the heat-equation the situation is relatively well-understood, 
this is not the case for the Schr\"{o}dinger equation since semiclassical analysis (see for instance \cite{zworski})
roughly says that for $\hbar\to0$  sufficiently concentrated solutions of the Schr\"{o}dinger equation move approximately along classical geodesics. Clearly semiclassical analysis breaks down on the singularity $Z$.

It is then natural to come back on the quantization procedure that permits to pass from the description of a free classical particle moving on a Riemannian manifold  to the corresponding Schr\"{o}dinger equation. 

This is a complicated subject that has no unique answer. The resulting evolution equation for quantum particles depends indeed on the chosen quantization procedure.

Most of coordinate invariant quantization procedures modify the quantum Hamiltonian by a correction term depending on the scalar curvature $R$. In dimension two, the scalar curvature is twice the Gaussian curvature $K$ and the modified Schr\"{o}dinger equation is of the form
$$
\ii\hbar\,\partial_t\psi(t,p)=\frac{\hbar^2}{2m}\Big(-\Delta+  c K(p)\Big)\psi(t,p),
$$
where $c\geq0$ is a constant.  Values given in the literature include:
\begin{itemize}
\item path integral quantization: $c=1/3$ and  $c=2/3$ in  \cite{driver-22}, $c=1/2$ in \cite{driver-20};
\item covariant Weyl quantization: $c\in[0,2/3]$ including conventional Weyl quantization ($c=0$) in \cite{fulling};
\item geometric quantization for a real polarization: $c=1/3$ in  \cite{driver-97};
\item finite dimensional approximations to Wiener Measures: $c=2/3$ in \cite{driver}\footnote{Notice that the Schr\"odinger equation that one finds in the references \cite{driver,driver-20,driver-22,fulling,driver-97} is $\ii\hbar\,\partial_t\psi(t,p)=\frac{\hbar^2}{m}\Big(-\frac{1}{2}\Delta+  c' R(p)\Big)\psi(t,p)$. Hence, $c=4c'$.}.
\end{itemize}

We refer to \cite{driver,fulling} for interesting discussions on the subject.\footnote{There are also other approaches to the quantization process on Riemannian manifolds that provide correction terms depending on the curvature. 
For instance  if one consider the Laplacian on a $\epsilon$-tubular neighborhood of a surface in ${\bf R}^3$ with Dirichlet boundary conditions, then for $\epsilon\to0$ after a suitable renormalization, one gets an operator containing a correction term depending on the Gaussian curvature and the square of the mean curvature (see \cite{David2,Lampart}).} 

Purpose of this paper is to study the self-adjointness of the {\em curvature Laplacian} $-\Delta+  c K$ depending on $c$ to understand if quantum confinement holds for the dynamics induced by this operator.
 Before stating our main result, let us remark that the curvature term $cK$ interacts with the diverging first order term in $\Delta$. 

For instance for the Grushin cylinder a unitary transformation (see Section \ref{subsec:proof1}, \eqref{U}) permits to transform the operator 

$$
\Delta=\dfrac{\partial^2}{\partial x^2}+x^2\dfrac{\partial^2}{\partial y^2}-\frac1x\dfrac{\partial}{\partial x}\mbox{ on }L^2\left({\bf R}\times S^1,\frac{1}{|x|}dx\,dy\right)
$$ in 
$$
\tilde\Delta=\dfrac{\partial^2}{\partial x^2}+x^2\dfrac{\partial^2}{\partial y^2}-\frac34\dfrac{1}{x^2}\mbox{ on } L^2({\bf R}\times S^1,dx\,dy)
$$
and hence the adding of a term of the form $-cK=-c\big(-2\frac1{x^2} \big)$ (that remains untouched by the unitary transformation) to $\tilde\Delta$  changes the diverging behaviour around ${x=0}$. In particular for $c=3/8$ the diverging potential disappears and $-\tilde\Delta+cK$ is not essentially self-adjoint in $L^2({\bf R}_+\times S^1,dx\,dy)$ while $-\tilde\Delta$ does. The same conclusion applies to $-\Delta+cK$ in $L^2({\bf R}_+\times S^1,\frac{1}{|x|}dx\,dy)$. \\

The main result of the paper is that the perturbation term given by the curvature destroys the essential self-adjointness of the Laplace-Beltrami operator.

\begin{theorem}\label{theorem}
Let $M$ be a compact oriented $2$-dimensional manifold equipped with a genuine 2-step $2$-ARS.  Let $\Omega$ be a connected component of $M\setminus Z$, where $Z$ is the singular set.  Let $g$ be the Riemannian metric induced by the $2$-ARS on $\Omega$, $\omega$ be the corresponding Riemannian area, $K$ the corresponding Gaussian curvature and $\Delta={\rm div}_\omega\circ {\rm grad}_g$ the Laplace-Beltrami operator.   Let $c \in [0,1/2)$. The curvature Laplacian  $-\Delta+cK$ with domain $C^\infty_0(\Omega)$, is essentially self-adjoint on $L^2(\Omega, \omega)$ if and only if c=0.
Moreover, if $c>0$, the curvature Laplacian has infinite deficiency indices.
\end{theorem}

The non-self-adjointness of $-\Delta+cK$  implies that one can construct self-adjoint extensions of this operator that permit to the solution to the Schr\"{o}dinger equation to flow out of the set $\Omega$, in the same spirit of \cite{Boscain-Prandi-JDE-2016,Gallone-Michelangeli-Pozzoli-2020}.
The study of these self-adjoint extension and how semiclassical analysis applies to them is a subject that deserves to be studied in detail.

\begin{remark}
We remark that in Theorem \ref{theorem} the hypothesis $c\in[0,1/2)$ guarantees the non-negativity of the operator $-\Delta+cK$ modulo a Kato-small perturbation. For further details, see also Remark \ref{rmk:conjecture}, Lemma \ref{small}, and Remark \ref{rmk:important}
\end{remark}

\begin{remark} Notice that in Theorem \ref{theorem} one can also consider the case $c< 0$. In this case, one can prove that the curvature Laplacian is essentially self-adjoint (applying for example the criterion for the self-adjointes of operators of the form $-\Delta+V$ on non-complete Riemannian manifolds, found in \cite{Prandi-Rizzi-Seri-2016}). However, if this case admits a physical interpretation is not known to the authors.
\end{remark}

As in Theorem \ref{prototheorem} the compactness hypothesis is useful to simplify the statement of the theorem. A version without the compactness hypothesis is given here where also the orientability assumption of $M$ is not needed.

\begin{theorem}
\label{posttheorem}
Let $M$ be a $2$-dimensional manifold equipped with a genuine 2-step $2$-ARS.  

Assume that 
\begin{itemize}
\item the singular set $Z$ is compact;
\item the $2$-ARS is geodesically complete. 
\end{itemize}
Let $\Omega$ be a connected component of $M\setminus Z$, and $c\in [0,1/2)$. With the same notations of Theorem \ref{theorem}, the curvature Laplacian  $-\Delta+cK$ with domain $C^\infty_0(\Omega)$, is essentially self-adjoint on $L^2(\Omega, \omega)$ if and only if c=0. Moreover, if $c>0$, the curvature Laplacian has infinite deficiency indices.
\end{theorem}
For the sake of simplicity, we prove Theorem \ref{theorem} only. Theorem \ref{posttheorem} can be proved following the same ideas.

Theorem \ref{posttheorem} applies in particular to the Grushin cylinder with curvature Laplacian $-\Delta+cK=-(\partial^2_x+x^2\partial^2_y-\frac1x\partial_x)+\frac{2c}{x^2}$. For this case, the fact that the deficiency indices are infinite means that all Fourier components of $-\Delta+cK$ are not self-adjoint.

Notice that under the hypothesis of the theorem, each connected component of $\partial\Omega$ is diffeomorfic to $S^1.$ 
Of course if $c>0$, the manifold does not need to be geodesically complete.

If one removes the 2-step hypothesis the situation is more complicated since tangency points \cite{ABS,euler} may appear.  In presence of tangency points even the essential self-adjointness of the standard Laplace-Beltrami operator (without the term $-cK$) is an open question \cite{Boscain-Laurent-2013}. Without the $2$-step hypothesis results can indeed be very different.  
To illustrate this, we study  the $\alpha$-Grushin cylinder for which computations can be done explicitly for every value of $c$.

\begin{proposition}\label{alpha-cyl}
Fix $\alpha\in{\bf R}$. On ${\bf R}\times S^1$ consider the generalized Riemannian structure for which an orthonormal frame is given by
$$
X_1(x,y)=\dfrac{\partial}{\partial x}, \qquad X_2^{(\alpha)}(x,y)=x^\alpha \dfrac{\partial}{\partial y},\mbox{ ~~~ here } x\in{\bf R},~~y\in S^1.
$$ 
Let $c\geq0$. On ${\bf R}_+\times S^1$ the structure is Riemannian with Riemannian area $\frac{1}{|x|^\alpha}dx\,dy$. Let  $-\Delta_\alpha+cK_\alpha$ be the curvature Laplacian with domain $C^\infty_0({\bf R}_+\times S^1)$ acting on $L^2({\bf R}_+\times S^1, \frac{1}{|x|^\alpha}dx\,dy)$. Denote by 
$$\alpha_{c,\pm}=\dfrac{(-2c+1)\pm 2\sqrt{(c-2+\sqrt{3})(c-2-\sqrt{3})}}{4c-1}.$$
\begin{align*}
\bullet &\text{ If } 0\leq c <  1/4,  &-\Delta_\alpha+cK_\alpha& \text{ is essentially self-adjoint if and only if } \alpha\geq \alpha_{c,+} \text{ or }\alpha\leq \alpha_{c,-};\\
\bullet &\text{ if } c=1/4, &-\Delta_\alpha+cK_\alpha& \text{ is essentially self-adjoint if and only if } \alpha\geq 3;\\
\bullet &\text{ if } 1/4<c\leq 2-\sqrt{3},  &-\Delta_\alpha+cK_\alpha& \text{ is essentially self-adjoint if and only if } \alpha_{c,-}\leq\alpha\leq \alpha_{c,+}; \\
\bullet &\text{ if }  2-\sqrt{3}<c< 2+\sqrt{3}, &-\Delta_\alpha+cK_\alpha& \text{ is not essentially self-adjoint } \forall \alpha \in {\bf R};\\
\bullet &\text{ if }  c\geq 2+\sqrt{3}, &-\Delta_\alpha+cK_\alpha& \text{ is essentially self-adjoint if and only if } \alpha_{c,-}\leq\alpha\leq \alpha_{c,+}.
\end{align*}
\end{proposition}

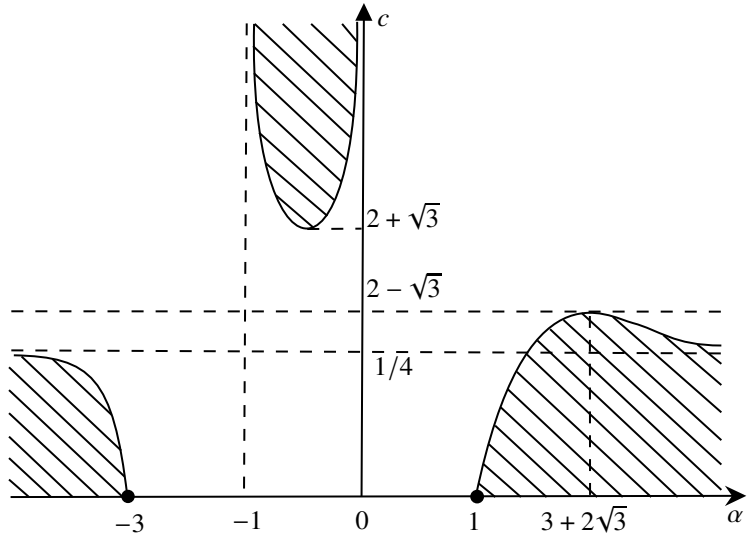
\begin{figure}[h]

\begin{center}

\tikzset{every picture/.style={line width=0.75pt}} 

\begin{tikzpicture}[x=0.75pt,y=0.75pt,yscale=-1,xscale=1]

\draw    (256.88,2.01) -- (255.7,247.53) ;
\draw [shift={(256.89,-0.99)}, rotate = 90.28] [fill={rgb, 255:red, 0; green, 0; blue, 0 }  ][line width=0.08]  [draw opacity=0] (8.93,-4.29) -- (0,0) -- (8.93,4.29) -- cycle    ;
\draw    (80.59,247.54) -- (444.69,246.94) ;
\draw [shift={(447.69,246.94)}, rotate = 539.9100000000001] [fill={rgb, 255:red, 0; green, 0; blue, 0 }  ][line width=0.08]  [draw opacity=0] (10.72,-5.15) -- (0,0) -- (10.72,5.15) -- (7.12,0) -- cycle    ;
\draw  [dash pattern={on 4.5pt off 4.5pt}]  (198.41,9.27) -- (197.04,248.53) ;
\draw  [dash pattern={on 4.5pt off 4.5pt}]  (80.96,173.75) -- (440.61,175.31) ;
\draw  [dash pattern={on 4.5pt off 4.5pt}]  (80.96,154.09) -- (435.22,154.27) ;
\draw    (202.09,9.51) .. controls (199.06,139.5) and (255.36,154.25) .. (253.5,8.47) ;
\draw    (312.77,246.94) .. controls (337.04,129.17) and (378.12,156.59) .. (391.14,160.03) .. controls (404.16,163.46) and (411.34,171.31) .. (435.22,171.31) ;
\draw    (82.21,176.23) .. controls (121.01,177.03) and (131.77,187.12) .. (138.82,246.94) ;
\draw  [dash pattern={on 4.5pt off 4.5pt}]  (228.61,112.71) -- (255.71,112.52) ;
\draw  [dash pattern={on 4.5pt off 4.5pt}]  (369.61,156.46) -- (369.71,249.02) ;
\draw    (79.81,227.94) -- (100.81,247.54) ;
\draw    (79.81,217.14) -- (111.61,247.54) ;
\draw    (79.81,205.34) -- (124.61,247.54) ;
\draw    (79.81,193.74) -- (136.61,247.54) ;
\draw    (79.81,238.74) -- (89.81,247.54) ;
\draw    (79.81,182.26) -- (137.5,237.29) ;
\draw    (86.07,176.83) -- (135.81,224.2) ;
\draw    (101.41,178) -- (131.01,206.2) ;
\draw    (328.61,193.26) -- (386.01,247.54) ;
\draw    (333.18,184.55) -- (400.21,247.54) ;
\draw    (338.18,175.8) -- (413.41,247.54) ;
\draw    (343.47,168.8) -- (426.61,247.54) ;
\draw    (349.75,163.09) -- (435.22,243.94) ;
\draw    (356.47,158.49) -- (435.22,231.2) ;
\draw    (365.61,155.06) -- (435.22,218.94) ;
\draw    (381.09,156.92) -- (435.22,204.94) ;
\draw    (406.23,166.06) -- (435.22,191.2) ;
\draw    (423.23,170.78) -- (435.22,180.54) ;
\draw    (325.18,202.43) -- (373.22,247.54) ;
\draw    (322.36,211.6) -- (360.17,246.94) ;
\draw    (319.18,220.92) -- (348.02,247.54) ;
\draw    (316.47,230.92) -- (335.22,247.54) ;
\draw    (314.04,240.63) -- (322.42,247.54) ;
\draw    (202.4,9.51) -- (251.83,56.31) ;
\draw    (201.6,20.71) -- (250.03,66.51) ;
\draw    (202,34.31) -- (248.14,77.88) ;
\draw    (202.8,48.71) -- (245.64,88.88) ;
\draw    (204.6,64.26) -- (242.03,98.06) ;
\draw    (207.8,79.51) -- (238.63,106.11) ;
\draw    (212.95,95.71) -- (231.36,111.71) ;
\draw    (216.4,9.51) -- (252.89,44.88) ;
\draw    (230.55,9.51) -- (253.14,30.88) ;
\draw    (246,9.51) -- (253.39,18.38) ;
\draw  [fill={rgb, 255:red, 0; green, 0; blue, 0 }  ,fill opacity=1 ] (310.43,247.32) .. controls (310.43,245.7) and (311.75,244.38) .. (313.37,244.38) .. controls (315,244.38) and (316.31,245.7) .. (316.31,247.32) .. controls (316.31,248.95) and (315,250.26) .. (313.37,250.26) .. controls (311.75,250.26) and (310.43,248.95) .. (310.43,247.32) -- cycle ;
\draw  [fill={rgb, 255:red, 0; green, 0; blue, 0 }  ,fill opacity=1 ] (136.33,247.23) .. controls (136.33,245.66) and (137.6,244.38) .. (139.18,244.38) .. controls (140.75,244.38) and (142.03,245.66) .. (142.03,247.23) .. controls (142.03,248.8) and (140.75,250.08) .. (139.18,250.08) .. controls (137.6,250.08) and (136.33,248.8) .. (136.33,247.23) -- cycle ;

\draw (251.2,253) node [anchor=north west][inner sep=0.75pt]    {$0$};
\draw (189.8,253) node [anchor=north west][inner sep=0.75pt]    {$-1$};
\draw (307,253.8) node [anchor=north west][inner sep=0.75pt]    {$1$};
\draw (436.67,252.73) node [anchor=north west][inner sep=0.75pt]    {$\alpha $};
\draw (262,3.07) node [anchor=north west][inner sep=0.75pt]    {$c$};
\draw (260.4,175.2) node [anchor=north west][inner sep=0.75pt]    {$1/4$};
\draw (254.4,97.4) node [anchor=north west][inner sep=0.75pt]  [font=\normalsize]  {\,$2+\sqrt{3} $};
\draw (254.4,133) node [anchor=north west][inner sep=0.75pt]  [font=\normalsize]  {\,$2-\sqrt{3}$};
\draw (343.6,249.94) node [anchor=north west][inner sep=0.75pt]    {$3+2\sqrt{3}$};
\draw (131,254.6) node [anchor=north west][inner sep=0.75pt]    {$-3$};

\end{tikzpicture}

\end{center}
\caption{Regions of the $(\alpha,c)$-parameter space where the operator $-\Delta_\alpha + cK_\alpha$ is essentially self-adjoint.}\label{region}
\end{figure}

The regions where $-\Delta_\alpha+cK_\alpha$ is essentially self-adjoint are plotted in Figure \ref{region}. Note that for some of the quantizations listed earlier, $-\Delta_\alpha+cK_\alpha$ is essentially self-adjoint for $|\alpha|$ sufficiently big. 
The $\alpha$-Grushin cylinder is an interesting geometric structure studied in \cite{boscain-neel,Boscain-Prandi-JDE-2016,Gallone-Michelangeli-Pozzoli-2020}. 
For $\alpha=0$ it is a flat cylinder, for $\alpha$ positive integer is a $(\alpha+1)$-step $2$-ARS; for $\alpha$ negative it describes a conic-like surface (in particular for $\alpha=-1$ describes a flat two-dimensional cone).

\begin{remark}\label{rmk:conjecture}
Notice that for $\alpha=1$ the $\alpha$-Grushin cylinder is the standard Grushin cylinder, for which in Proposition \ref{alpha-cyl} we obtain that $-\Delta_1+cK_1$, $c\geq0$, is essentially self-adjoint if and only if $c=0$. This suggests that the hypothesis $c\in [0,1/2)$ of Theorem \ref{theorem} is technical and that the same result could be extended to the range of values $c\geq 0$.
\end{remark}

The proof of Proposition \ref{alpha-cyl}, which is instructive since it is simple and presents already some crucial ingredients necessary for the general theory, is given in Section \ref{sec:essop}.

\medskip

\medskip\noindent {\bf Structure of the paper.} In Section \ref{s-ARS} we give the key definition and results for $2$-ARS. In Section \ref{sec:essop} we introduce the basic concepts to study the self-adjointness of symmetric operators and give a proof of Proposition \ref{alpha-cyl}. The proof of Theorem \ref{theorem} spans Sections \ref{sec:grushinzone} and \ref{sec:proofonmanifold}. A local version around a singular region is studied in Section \ref{sec:grushinzone}, where a description of the closure and adjoint curvature Laplacian operators is given. The main tools needed for our proof are the Hardy inequality, the Kato-Rellich perturbation theorem, the Fourier transform, and Sturm-Liouville theory applied here in the context of 2D operators. We then extend the results on the whole manifold in Section \ref{sec:proofonmanifold}.

\medskip

We conclude this introduction by remarking that while an operator of the form 
$-\Delta +  c K(p)$ is useful to describe a quantum particle in a Riemannian manifold, it is 
not  meaningful in the description of the evolution of the heat. 
Indeed a heat equation of the form $\partial_t\phi=  (-\Delta+ c K(p) )\phi  $ would describe the evolution of a random particle on a Riemannian manifold with a rate of killing proportional to the Gaussian curvature.

\section{2D almost-Riemannian structures}
\label{s-ARS}

\begin{definition}
\label{d-ugo-ar}
Let $M$ be a 2D connected smooth manifold. 
A $2$-dimensional almost-Riemannian Structure ($2$-ARS) on $M$ is a pair $(\cb,f)$ as follows:
\begin{itemize}
\item[1.] $\cb$ is an Euclidean bundle over $M$ of rank $2$. We denote each fiber by $\fibra_q$, the scalar product on $\fibra_q$ by $\metrup{\cdot}{\cdot}{q}$ and the norm of $u\in\fibra_q$ as $|u|=\sqrt{\metrup{u}{u}{q}}.$

\item[2.] $f:\cb\to TM$ is a smooth map that is a morphism of vector bundles i.e., $f(U_q)\subseteq T_qM$ and $f$ is linear on fibers.

\item[3.] the {\em distribution} $\distr=\{f(\sigma)\mid\sigma:M\to\cb$ smooth section$\}$, is a family of vector fields satisfying the H\"ormander condition, i.e., defining $\distr_1:=\distr$, $\distr_{i+1}:=\distr_i+[\distr_1,\distr_i]$, for $i\geq 1$, there exists $s\in{\bf N}$ such that $\distr_s(q)=T_q M$.

\end{itemize}
\end{definition}
A particular case of 2-ARSs is given by Riemannian surfaces. In this case $\cb=TM$ and $f$ is the identity.

Let us recall few key definitions and facts. We refer to \cite{ABB} for more details.
\begin{itemize}
\item 
Let $\distr_p=\{X(p)\mid X\in\distr\}=f(U_p)\subseteq T_pM$.
The set of points in $M$ such that $\dim(\distr_p)<2$ is called {\it singular set} and it is denoted by $Z$.
Since $\distr$  satisfies the H\"ormander condition, the subspace $\distr_p$ is nontrivial for every $p$ and $Z$ coincides with the set of points $p$ where $\distr$ is one-dimensional. The $2$-ARS is said to be {\em genuine} if $Z\neq\emptyset$. The $2$-ARS is said to be {\em $2$-step} if for every $p\in M$ we have $\mathcal{D}_p+[\mathcal{D},\mathcal{D}]_p=T_pM$.

\item The {\em (almost-Riemannian) norm}  of a vector $v\in \distr_p$ is 
$$\norm{v}:= \min \{|u|, \, u\in \fibra_{p}\  \text{ s.t. }\  v=f(p,u)\}.$$
\item An {\em admissible curve} is a Lipschitz curve $\gamma:[0,T]\to M$ such that  there exists a measurable and essentially bounded function 
$u: [0,T]\ni t \mapsto u(t)\in \fibra_{\gamma(t)}$,
called   \emph{control function}, such that 
$
\dot \gamma(t)=f(\gamma(t),u(t)),$ for a.e. $t\in [0,T]$. Notice there may be more than one control corresponding to the same admissible curve.

\item If  $\gamma$ is admissible then $t\to\norm{\dot{\gamma}(t)}$ is measurable.
 The {\em (almost-Riemannian) length} of an admissible curve $\gamma:[0,T]\to M$  is  
$$\ell(\gamma):=\int_0^T\norm{\dot{\gamma}(t)} dt.$$
\item The {\em (almost-Riemannian) distance} between two points $p_{0},p_{1}\in M$ is  
$$
d(p_0,p_1)= \inf \{\ell(\gamma)\,|\,  \gamma:[0,T]\to M \text{ admissible},\  \gamma(0)=p_0,\  \gamma(T)=p_1\}.
$$
Thanks to the bracket-generating condition, the Chow-Rashevskii theorem
 guarantees that 
$(M,d)$ is a metric space and that the topology induced by $(M,d)$ is equivalent to the manifold topology.

\item Given a  local trivialization $\Omega\times {\bf R}^{2}$ of $\cb$,  an \emph{orthonormal frame} for the $2$-ARS on  $\Omega$ is the pair  of vector fields $\{F_1,F_2\}:=\{f\circ\sigma_1,f\circ\sigma_2\}$
where $\{\sigma_1,\sigma_2\}$ is an orthonormal frame for $\metrup{\cdot}{\cdot}{q}$ on $\Omega\times {\bf R}^{2}$ of $\cb$.
On a local trivialization  the map $f$ can be written as 
 $f(p,u)=u_1F_1(p)+u_2F_2(p)$. When this can be done globally (i.e., when $\cb$ is  the trivial bundle) we say that the $2$-ARS  is {\em free}.

Notice that orthonormal frames in the sense above are orthonormal frames  in the Riemannian sense out of the singular set.

\item Locally, for a $2$-ARS,  it is always possible to find a system of coordinates and an orthonormal frame that in these coordinates has the form
\begin{eqnarray}
F_1(x,y)=\left(
\begin{array}{c}
1\\0
\end{array}
\right),~~
F_2(x,y)=\left(
\begin{array}{c}
0\\\ff(x,y)
\end{array}
\right),
\label{eq-100f}
\end{eqnarray}
where $\ff:\Omega\to{\bf R}$ is a smooth function. In these coordinates we have that $Z=\{(x,y)\in\Omega\mid \ff(x,y)=0\}$. Using this orthonormal frame one immediately gets:
\begin{proposition}
The $2$-ARS is $2$-step in $\Omega$ if and only if for every 
 $(x,y)\in\Omega$  such that $\ff(x,y)=0$, we have $\partial_x\ff(x,y)\neq0$.
\end{proposition}

Moreover, the implicit function theorem applied to the function $\ff$ directly implies:\begin{proposition}
If the  $2$-ARS is genuine and $2$-step then $Z$ is a closed embedded one dimensional submanifold.
\end{proposition}
In particular if  $Z$ is compact, each connected component of  $Z$ is diffeomorphic to $S^1$.

\item Out of the singular set $Z$, the structure is Riemannian and the Riemannian metric, the Riemannian area, the Riemannian curvature, and the Laplace-Beltrami operator are easily expressed in the orthonormal frame given by \eqref{eq-100f}:

 \begin{eqnarray}
g_{(x,y)}=\left(
\begin{array}{cc}
1&0\\0&\frac{1}{\ff(x,y)^2}
\end{array}
\right),\label{e-ugo-1/f2}
\end{eqnarray}
 \begin{eqnarray}
\omega_{(x,y)}=\frac{1}{|\ff(x,y)|}dx\,dy,\label{e-ugo-dA}
\end{eqnarray}
 \begin{eqnarray}
K(x,y)=   \frac{\ff(x,y) \partial_{x}^2  \ff(x,y)-2
   \left( \partial_{x}\ff(x,y) \right)^2}{\ff(x,y)^2},\label{e-ugo-k}\\
\Delta=\partial_{x}^2+\ff^2 \partial_{y}^2 -
\frac{\partial_{x}\ff}{\ff}\partial_{x}
+\ff(\partial_{y}\ff) \partial_{y}.\label{e-ugo-D}
\end{eqnarray}

\item To prove the main results of this paper, the following normal forms are going to be important.
\begin{proposition}[\cite{ABS}]\label{normalform}
Consider a $2$-step $2$-ARS. For every $p\in M$ there exist a neighborhood $U$ of $p$, a system of coordinates in $U$, and an orthonormal frame $\{X_1,X_2\}$ for the ARS on $U$, such that $p=(0,0)$ and $\{X_1,X_2\}$ has one of the following forms:
\begin{enumerate}
\item [F1.] $X_1(x,y)=\dfrac{\partial}{\partial x}, \qquad X_2(x,y)=e^{\phi(x,y)}\dfrac{\partial}{\partial y},$
\item [F2.] $X_1(x,y)=\dfrac{\partial}{\partial x}, \qquad X_2(x,y)=xe^{\phi(x,y)}\dfrac{\partial}{\partial y},$
\end{enumerate}
where $\phi$ is a smooth function such that $\phi(0,y)=0$.
 \end{proposition}
 A point $p\in M$ is said to be a \emph{Riemannian point} if $\mathcal{D}_p$ is two-dimensional, and hence a local description around $p$ is given by F1. A point $p$ such that $\mathcal{D}_p$ is one-dimensional, and thus $\mathcal{D}_p+[\mathcal{D},\mathcal{D}]_p$ is two-dimensional, is called a \emph{Grushin point} and a local description around $p$ is given by F2. 
 
When $M$ is compact orientable, each connected component of $Z$ is diffeomorphic to $S^1$ and admits a tubular neighborhood diffeomorphic to  ${\bf R}\times S^1$. In this case  the normal form F2 can be extended to the whole neighborhood.

\begin{proposition}[\cite{Boscain-Laurent-2013}]
\label{p-globalizziagite}
 Consider a $2$-step $2$-ARS on a compact  orientable manifold. Let $W$ be a connected component of $Z$. Then there exist a tubular neighborhood $U$ of $W$  diffeomorphic to ${\bf R}\times S^1$, a system of coordinates in $U$, and 
 an orthonormal frame $\{X_1,X_2\}$ of the 2-ARS on $U$ such that
 $W=\{(0,y),~y\in S^1\}$ and $\{X_1,X_2\}$ has the form
 \begin{equation}
\label{global-grushin}
 X_1(x,y)=\dfrac{\partial}{\partial x}, \qquad X_2(x,y)=xe^{\phi(x,y)}\dfrac{\partial}{\partial y}.
\end{equation}
\end{proposition}

\end{itemize}

\section{Self-adjointness of operators}
\label{sec:essop}
Let $A$ be a linear operator on a separable Hilbert space $\mathcal{H}$, $(\cdot,\cdot)_{\mathcal{H}}$. The linear subspace of $\mathcal{H}$ where the action of $A$ is well-defined is called the domain of $A$, denoted by $\mathcal{D}(A)$. We shall always assume that $\mathcal{D}(A)$ is dense in $\mathcal{H}$. Following \cite{rs2}, we recall several definitions and properties of linear operators:
\begin{itemize}
 \item $A$ is said to be \emph{symmetric} if $(A u, v)_{\mathcal{H}}=(u, Av)_{\mathcal{H}}$ for all $u,v\in\mathcal{D}(A)$.
 \item $A$ is said to be \emph{closed} if $\mathcal{D}(A)$ with the norm $\|\cdot\|_A:=\|\cdot\|_{\mathcal{H}}+\|A\cdot\|_{\mathcal{H}}$ is complete.
 \item A linear operator $B$, $\mathcal{D}(B)$ such that $\mathcal{D}(A)\subset\mathcal{D}(B)$ and $Bu=Au$ for all $u\in \mathcal{D}(A)$ is called an \emph{extension} of $A$. In this case we write $A\subset B$.
 \item If $A$ is symmetric and densely defined, there exists a minimal closed symmetric extension $\overline{A}$ of $A$, which is said to be the \emph{closure} of $A$. We describe this construction: take any sequence $(u_n)_n\subset \mathcal{D}(A)$ which converges to a limit $u\in\mathcal{H}$, and for which the sequence $(Au_n)_n$ converges to a limit $w\in\mathcal{H}$. Then, by symmetry of $A$, we have that 
$$(w,v)_{\mathcal{H}}=\lim_{n\to\infty} (Au_n,v)_{\mathcal{H}}=\lim_{n\to\infty} (u_n,Av)_{\mathcal{H}}=(u,Av)_{\mathcal{H}}, \quad \forall v\in\mathcal{D}(A).$$
Since $\mathcal{D}(A)$ is dense in $\mathcal{H}$, $w$ is uniquely determined by $u$. The closure of $A$ is defined by setting $\overline{A}u=w$, and the domain $\mathcal{D}(\overline{A})$ is the closure of $\mathcal{D}(A)$ with respect to the norm $\|\cdot\|_A$. One can easily see that $\overline{A}$ is closed, symmetric, and any closed extension of $A$ is an extension of $\overline{A}$ as well.
\item Given a densely defined linear operator $A$, the domain $\mathcal{D}(A^*)$ of the adjoint operator $A^*$ is the set of all $v\in\mathcal{H}$ such that there exists $w\in\mathcal{H}$ with $(Au,v)_{\mathcal{H}}=(u,w)_{\mathcal{H}}$ for all $u\in\mathcal{D}(A)$. The adjoint of $A$ is defined by setting $A^*v=w$. 
\item $A$ is said to be \emph{self-adjoint} if $A^*=A$, that is, $A$ is symmetric and $\mathcal{D}(A^*)=\mathcal{D}(A)$. 
\item $A$ is said to be \emph{essentially self-adjoint} if its closure is self-adjoint.
\item If $B$ is a closed symmetric extension of $A$, then $A\subset\overline{A}\subset B\subset A^*$.
\item If a densely defined operator $B$ is such that $ \mathcal{D}(A)\subset \mathcal{D}(B)$ and there exist $a,b\geq 0$ such that
\begin{equation}
\label{eq:estim}
\|B u\|\leq a\|A u\|+b\|u\|,\qquad \forall u\in \mathcal{D}(A),
\end{equation}
then $B$ is said to be \emph{small with respect to} $A$ (or also, Kato-small w.r.t. $A$). The infimum of the set of $a\geq 0$ such that \eqref{eq:estim} holds is called the $A$-bound of $B$. If $a$ can be chose arbitrarily small, $B$ is said to be \emph{infinitesimally} small w.r.t $A$. 
\end{itemize}

We will need the following classical result in perturbation theory:
\begin{proposition}[Kato-Rellich's Theorem]
\label{lemma:Kato}
Let $A,B$ be two densely defined operators and assume that $B$ is small with respect to $A$. Then $D(\overline{A})\subset D(\overline{A+B})$. If moreover $a<1$ in~\eqref{eq:estim}, then $D(\overline{A})= D(\overline{A+B})$.
\end{proposition}

In order to study the self-adjointness and more in general describe the extensions of a symmetric operator $A$, one may use the fundamental Von Neumann decomposition (\cite[Chapter X]{rs2})
\begin{equation}\label{VNdecomposition}
\mathcal{D}(A^*)=\mathcal{D}(\overline{A})\oplus_A \ker(A^*+\mathrm{i})\oplus_A  \ker(A^*-\mathrm{i}),
\end{equation}
where the sum is orthogonal with respect to the scalar product $(\cdot,\cdot)_A=(\cdot,\cdot)_{\mathcal{H}}+(A^*\cdot,A^*\cdot)_{\mathcal{H}}.$
As a first direct consequence of \eqref{VNdecomposition}, one has the fundamental spectral criterion for self-adjointness:
\begin{proposition}\label{key}
Let $A$, $\mathcal{D}(A)$, be a symmetric operator, densely defined on the Hilbert space $\mathcal{H}$. The following are equivalent:
\begin{enumerate}
\item[(a)] $A$ is essentially self-adjoint;
\item[(b)]$\mathrm{Ran}(A\pm \ii)$ is dense in $\mathcal{H}$;
\item[(c)] $\ker(A^*\pm \ii)=\{0\}$.
\end{enumerate}
\end{proposition}
\begin{itemize} 
\item The dimensions of the vector spaces $\ker(A^*+ \ii)$ and $\ker(A^*- \ii)$ are called \emph{deficiency indices} of $A$.
\end{itemize} 
Always using \eqref{VNdecomposition}, one can deduce that $A$ admits self-adjoint extensions if and only if its deficiency indices are equal (\cite[Corollary to Theorem X.2]{rs2}).\\

Another immediate consequence of \eqref{VNdecomposition} is the following fact concerning 1D operators: let $A$ be a 1D Sturm-Liouville operator $-\frac{d^2}{dx^2}+V(x)$, where $V$ is a continuous real function on ${\bf R}_+$, acting on $L^2({\bf R}_+,dx)$ with domain $C^\infty_0({\bf R}_+)$ (for a general introduction to 1D Sturm-Liouville operators, see e.g. \cite[Chapter 15]{schmu_unbdd_sa}). Then, since the eigenvalue equation
$$
-u''(x) +V(x)u(x)=\pm \mathrm{i}\, u(x) 
$$
has always two linearly independent solutions, the quotient $\mathcal{D}(A^*)/\mathcal{D}(\overline{A})$ has at most dimension four. Moreover, let us recall the \emph{limit point-limit circle} Weyl's Theorem (see, e.g, \cite[Appendix to Chapter X.1]{rs2}) which says that the self-adjointness of a 1D Sturm-Liouville operator can be deduced by regarding the solutions to the ODE
\begin{equation}\label{lp-lc}
-u''(x)+V(x)u(x)=0. 
\end{equation}
\begin{itemize}
\item If all solutions to \eqref{lp-lc} are square-integrable near $0$ (respectively $\infty$), then $V$ is said to be in the limit circle case at $0$ (resp. $\infty$). If $V$ is not in the limit circle case at $0$ (resp. $\infty$), it is said to be in the limit point case at $0$ (resp. $\infty$).
\end{itemize}
\begin{proposition}[Weyl's Theorem]\label{weyl}
The operator $-\frac{d^2}{dx^2}+V(x)$ with domain $C^\infty_0({\bf R}_+)$ has deficiency indices 
\begin{itemize}
\item $(2,2)$ if $V$ is in the limit circle case at both $0$ and $\infty$;
\item $(1,1)$ if $V$ is in the limit circle case at one end point and in the limit point at the other;
\item $(0,0)$ if $V$ is in the limit point case at both $0$ and $\infty$.
\end{itemize}
In particular, $-\frac{d^2}{dx^2}+V(x)$ is essentially self-adjoint on $L^2({\bf R}_+,dx)$ if and only if $V$ is in the limit point case at both $0$ and $\infty$. 
\end{proposition}
Some useful criteria to determine whether a potential $V$ is in the limit point or limit circle case (it is also said to be quantum-mechanically complete or incomplete, respectively) at $0$ and $\infty$ are the following:
\begin{proposition}\label{completeinfty}
Let $V\in C^1({\bf R}_+)$ be real and bounded above by a constant $E$ on $[1,\infty)$. Suppose that $\int_1^\infty \frac{1}{\sqrt{E-V(x)}}dx=\infty$ and $V'/|V|^{3/2}$ is bounded near $\infty$. Then $V$ is in the limit point case at~$\infty$. 
\end{proposition}
\begin{proposition}\label{complete0}
Let $V\in C^0({\bf R}_+)$ be real and positive near $0$. If $V(x)\geq \frac{3}{4x^2}$ near $0$ then $V$ is in the limit point case at $0$. If for some $\epsilon>0$, $V(x)\leq (\frac{3}{4}-\epsilon)\frac{1}{x^2}$ near $0$, then $V$ is in the limit circle case at $0$.
\end{proposition}
\begin{proposition}\label{noncomplete0}
Let $V\in C^0({\bf R}_+)$ be real, and suppose that it decreases as $x \downarrow 0$. Then $V$ is in the limit circle case at $0$.
\end{proposition}
Weyl's Theorem and these criteria are, respectively, \cite[Theorem X.7, Corollary to Theorem X.8, Theorem X.10 and Problem X.7 ]{rs2}.

Here we give the proof of Proposition \ref{alpha-cyl}, which makes use of the limit point-limit circle argument.
\begin{proof}\emph{of Proposition \ref{alpha-cyl}}\;\;
The Laplace-Beltrami operator (with domain $C^\infty_0({\bf R}_+\times S^1)$) and the curvature associated to the orthonormal frame $X_1,X_2^{(\alpha)}$ are given by 
$$\Delta_\alpha=\frac{\partial^2}{\partial x^2}+x^{2\alpha}\frac{\partial^2}{\partial y^2}-\frac{\alpha}{x}\frac{\partial}{\partial x},\quad K_\alpha=-\frac{\alpha(\alpha+1)}{x^2}.$$
We perform a unitary transformation
$$
U_\alpha:L^2\Big({\bf R}_+\times S^1, \frac{1}{|x|^\alpha}dx\,dy\Big) \rightarrow  L^2({\bf R}_+\times S^1, dx\,dy), \quad \psi \mapsto |x|^{-\alpha/2}\psi,
$$
which gives the operator
$$L_{\alpha,c}:=U_\alpha (-\Delta_\alpha+cK_\alpha)U_\alpha^{-1}=-\frac{\partial^2}{\partial x^2}-x^{2\alpha}\frac{\partial^2}{\partial y^2}+\frac{(1-4c)\alpha^2+(2-4c)\alpha}{4x^2}.$$
Via Fourier transform with respect to the variable $y\in S^1$, one obtains the direct sum operator
$$\widetilde{L}_{\alpha,c}:=\bigoplus_{k\in {\bf Z}}(L_{\alpha,c})_k,\quad (L_{\alpha,c})_k:=-\frac{\partial^2}{\partial x^2}+x^{2\alpha}k^2+\frac{(1-4c)\alpha^2+(2-4c)\alpha}{4x^2}, $$
acting on the Hilbert space $\ell^2(L^2({\bf R}_+))$, with domain $\mathcal{D}(\widetilde{L}_{\alpha,c})=\{(f_k)_{k\in{\bf Z}}\in\ell^2(L^2({\bf R}_+))\mid f_k\in \mathcal{D}((L_{\alpha,c})_k)\;\forall k\in{\bf Z}, f_k=0 \text{ for almost every $k$}\}$, where $\mathcal{D}\Big((L_{\alpha,c})_k\Big)=C^\infty_0({\bf R}_+)$ for all $k$. Moreover, as a general fact concerning direct sum operators, $\widetilde{L}_{\alpha,c}$ is essentially self-adjoint if and only if $(L_{\alpha,c})_k$ is so for all $k$ (\cite[Proposition 2.3]{Boscain-Prandi-JDE-2016}). Let us thus study the essential self-adjointness of $(L_{\alpha,c})_0$: it is a Sturm-Liouville operator of the form $-\frac{d^2}{dx^2}+V_{\alpha,c}(x)$, where 
\begin{equation}\label{inversesquare}
V_{\alpha,c}(x)=\frac{(1-4c)\alpha^2+(2-4c)\alpha}{4x^2}=:\frac{k(\alpha,c)}{x^2}.
\end{equation}
The potential $V_{\alpha,c}$ is quantum-mechanically complete at infinity, for all $(\alpha,c)\in{\bf R}\times[0,\infty)$ (as one can check by applying Proposition \ref{completeinfty}). So, applying Propositions \ref{complete0} and \ref{noncomplete0}, we can conclude that $2(L_{\alpha,c})_0$ is essentially self-adjoint if and only if $V_{\alpha,c}(x)\geq \frac{3}{4x^2}$ near zero, since $V_{\alpha,c}(x)=k(\alpha,c)/x^2$ and when $k(\alpha,c)<0$ then $V_{\alpha,c}$ decreases for $x\downarrow 0$. By using the explicit formula \eqref{inversesquare} for $V_{\alpha,c}$, we obtain $k(\alpha,c)\geq 3/4$ which yields the values of $(\alpha,c)$ given in the statement. For what we said before, when $(L_{\alpha,c})_0$ is not essentially self-adjoint, neither is so $\widetilde{L}_{\alpha,c}$. Finally, when $(L_{\alpha,c})_0$ is essentially self-adjoint, then 
$$x^{2\alpha}k^2+\frac{(1-4c)\alpha^2+(2-4c)\alpha}{4x^2}\geq \frac{3}{4x^2}, $$
so $(L_{\alpha,c})_k$ is essentially self-adjoint too, for all $k$, and hence $\widetilde{L}_{\alpha,c}$ is so.\hfill$\Box$\end{proof}

\begin{remark}
As a by-product of the proof of Proposition \ref{alpha-cyl} we obtain that for the Grushin cylinder ($\alpha=1$) with $c>0$ all Fourier components of $-\Delta_1+cK_1$ are not essentially self-adjoint, due to the inequality $x^2k^2+(\frac{3}{4}-2c)\frac{1}{x^2}<\frac{3}{4x^2}$ which holds near zero for all $k\in {\bf Z}$. Hence $-\Delta_1+cK_1$ has infinite deficiency indices, for any $c>0$. When $c\in(0,1/2)$, Theorem \ref{theorem} extends this result to any two-dimensional almost Riemannian manifold of step $2$ (under some natural topological assumptions).

 However, from this proof we can also point out that this is not always the case for the $\alpha$-Grushin cylinder, as for instance in the case of a flat cone ($\alpha=-1$) with $c\geq0$ (note that this is not an ARS). In that case, the inequality $(k^2-\frac{1}{4})\frac{1}{x^2}\geq \frac{3}{4x^2}$ implies that the $k$-th's Fourier components of $-\Delta_{-1}+cK_{-1}$ are essentially self-adjoint for all $k\neq 0$, even if $(-\Delta_{-1}+cK_{-1})_0$ is not. This means that $-\Delta_{-1}+cK_{-1}$ has deficiency indices equal to $1$, for all $c\geq 0$.
 \end{remark}

We conclude this Section by considering a Riemannian manifold $(M,g)$ without boundary, with associated Riemannian volume form $\omega$, and Laplace-Beltrami operator $\Delta={\rm div}_\omega\circ {\rm grad}_g$ acting on the Hilbert space $L^2(M,\omega)$, with domain $\mathcal{D}(\Delta)=C^\infty_0(M)$. Green's identity implies
\begin{enumerate}
\item[(i)] $\int_M \overline{u}\;\Delta v \; d\omega=\int_M \overline{\Delta u}\; v \; d\omega$, for all $u,v\in C^\infty_0(M)$, i.e., $\Delta$ is a symmetric operator.
\item[(ii)] $\mathcal{D}(\Delta^*)=\{u\in L^2(M,\omega) \mid \Delta u\in L^2(M,\omega) \text{ in the sense of distributions} \}. $
\end{enumerate}
Letting $F$ be a real-valued continuous function locally $L^2(M,\omega)$ seen as a multiplicative operator with domain $C^\infty_0(M)$, (i) and (ii) still hold true for the operator $\Delta+F$ instead of $\Delta$. 
\begin{remark} Being $\Delta+F$ a real operator (that is, it commutes with complex conjugation), its deficiency indices are equal (\cite[Theorem X.3]{rs2}) and thus it always admits self-adjoint extensions.
\end{remark}

 \section{Grushin zone}
 \label{sec:grushinzone}

We focus now our attention around a Grushin point. We thus define the Riemannian manifold $\Omega=\{(x,y)\in {\bf R}\times S^1 \mid x\neq 0\}$ with metric $g=\mathrm{diag}(1,x^{-2}e^{-2\phi(x,y)})$, where $\phi$ is a smooth function which is constant for large $|x|$ . The smoothness of $\phi$ is guaranteed by Proposition \ref{normalform}, and even if it is only defined locally, we extend it constantly in the coordinates $(x,y)$, since what matters is the analysis close to $x=0$. Note that $X_1,X_2$ given by Proposition \ref{normalform} (F2) is an orthonormal frame for $g$. We then consider also the two connected components $\Omega_+=\{(x,y)\in {\bf R}\times S^1 \mid x> 0\}$ and $\Omega_-=\{(x,y)\in {\bf R}\times S^1 \mid x< 0\}$. We start by proving the following key result:
\begin{theorem}\label{Grushin}
Consider the Riemannian manifold $(\Omega_+,g)$, with associated Riemannian volume form $\omega$, curvature $K$ and Laplace-Beltrami operator $\Delta$. Let $c\in (0,1/2)$. Consider the curvature Laplacian $-\Delta+cK$, with domain $C_0^\infty(\Omega_+)$, acting on $L^2(\Omega_+, \omega)$. Then for every $\epsilon >0$ there exist functions $h_{\pm,\epsilon,c}\in L^2(\Omega_+, \omega)\cap C^\infty(\Omega_+)$ such that
\begin{enumerate}
\item [(i)] $h_{\pm,\epsilon,c} \in \mathcal{D}((-\Delta+cK)^*)$;
\item [(ii)]$h_{\pm,\epsilon,c} \notin \mathcal{D}(\overline{-\Delta+cK})$;
\item [(iii)] $\mathrm{supp}(h_{\pm,\epsilon,c})\subset (0,\epsilon)\times S^1$.
\end{enumerate}
In particular, $-\Delta+cK$ is not essentially self-adjoint (here $c\neq 0$). The same conclusions hold if we replace $\Omega_+$ with $\Omega_-$ or $\Omega$.
\end{theorem}
What we can actually prove is the following stronger version of Theorem \ref{Grushin}:
\begin{theorem}\label{infinitedeficency}
With the same assumptions and notations of Theorem \ref{Grushin}, $\dim \mathcal{D}((-\Delta+cK)^*)/\mathcal{D}(\overline{-\Delta+cK}) = \infty$, i.e., $-\Delta+cK$ has infinite deficiency indices.
\end{theorem}
The proofs of Theorems \ref{Grushin} and \ref{infinitedeficency} span Sections \ref{subsec:proof1} and  \ref{subsec:proof2}, where we shall describe respectively the closure and the adjoint of $-\Delta+cK$.

\subsection{Closure operator}
 \label{subsec:proof1}
We shall work on the manifold $\Omega_+$, being the case $\Omega_-$ analogous. Then the statement for $\Omega$ follows from the decomposition $L^2(\Omega, \omega)=L^2(\Omega_-, \omega)\oplus^\perp L^2(\Omega_+, \omega)$.\\

For a metric $g$ of the form $\mathrm{diag}(1,\ff(x,y)^{-2})$, plugging $\ff(x,y)=xe^{\phi(x,y)}$ into \eqref{e-ugo-dA}, \eqref{e-ugo-k}, and \eqref{e-ugo-D} one has the following:
 \[
 \begin{split}
 &\omega_{(x,y)}=\dfrac{1}{xe^{\phi(x,y)}}dxdy, \\ & 
 K(x,y)=-\dfrac{2}{x^2}-\dfrac{2\partial_ x\phi(x,y)}{x}+\partial_x^2\phi(x,y)-(\partial_x\phi(x,y))^2, \\ &
 \Delta=\partial_x^2+x^2e^{2\phi(x,y)}\partial_y^2-\dfrac{1}{x}\partial_x-\partial_x\phi(x,y)\partial_x+\partial_y\phi(x,y)x^2e^{2\phi(x,y)}\partial_y.
 \end{split}
 \]
We perform a unitary transformation
\begin{equation}\label{U}
U:L^2(\Omega_+,\omega) \rightarrow  L^2(\Omega_+,dxdy), \quad \psi \mapsto (xe^{\phi})^{-1/2}\psi,
\end{equation}
and the corresponding transformed Laplacian is given by 
\[
\begin{split}
L&=U\Delta U^{-1}\\ & =\partial_x^2+x^2e^{2\phi}\partial_y^2+2x^2e^{2\phi}(\partial_y\phi)\partial_y-\dfrac{3}{4x^2}-\dfrac{\partial_x\phi}{2x}-\dfrac{1}{4}(\partial_x\phi)^2+\dfrac{1}{2}\partial_x^2\phi+\dfrac{3}{4}x^2(\partial_y\phi)^2e^{2\phi}+\dfrac{1}{2}x^2(\partial^2_y\phi) e^{2\phi}.
\end{split}
\]
We shall analyze the self-adjointness of the operator
\begin{align}
-L+cK& =U(-\Delta+cK) U^{-1}\nonumber \\ & =- \partial_x^2-x^2e^{2\phi}\partial_y^2-2x^2e^{2\phi}(\partial_y\phi)\partial_y+\Big(\dfrac{3}{4}-2c\Big)\dfrac{1}{x^2}+\dfrac{1-4c}{2}\dfrac{\partial_x\phi}{x}+\Big(\dfrac{1}{4}-c\Big)(\partial_x\phi)^2\nonumber \\ &+\Big(c-\dfrac{1}{2}\Big)\partial_x^2\phi-\dfrac{3}{4}x^2(\partial_y\phi)^2e^{2\phi}-\dfrac{1}{2}x^2(\partial^2_y\phi) e^{2\phi}\nonumber \\ &
=H_c+\eta_c, \label{unitarylaplacian}
\end{align}
where we have defined the operator
$$
H_c=-\partial_x^2-x^2e^{2\phi}\partial_y^2-2x^2e^{2\phi}(\partial_y\phi)\partial_y+\Big(\dfrac{3}{4}-2c\Big)\dfrac{1}{x^2}+\Big(\dfrac{1-4c}{2}\Big)\dfrac{\partial_x\phi}{x},
$$
and the multiplicative operator 
$$
\eta_c=\Big(\dfrac{1}{4}-c\Big)(\partial_x\phi)^2+\Big(c-\dfrac{1}{2}\Big)\partial_x^2\phi-\dfrac{3}{4}x^2(\partial_y\phi)^2e^{2\phi}-\dfrac{1}{2}x^2(\partial^2_y\phi) e^{2\phi},
$$
both with domain $C_0^\infty(\Omega_+)$, acting on the Hilbert space $L^2(\Omega_+,dxdy)=:L^2(\Omega_+)$. For later convenience, we also define the operator 
$$T_c=-\partial_x^2-x^2e^{2\phi}\partial_y^2-2x^2e^{2\phi}(\partial_y\phi)\partial_y+\Big(\dfrac{3}{4}-2c\Big)\dfrac{1}{x^2},\quad \mathcal{D}(T_c)=C^\infty_0(\Omega_+),\quad \text{acting on } L^2(\Omega_+),
  $$
and the 1D operator, usually called inverse square potential or Bessel operator in the literature
\begin{equation}\label{1D}
s_c=-\dfrac{d^2}{dx^2}+\Big(\frac34-2c\Big)\dfrac{1}{x^2},\quad \mathcal{D}(s_c)=C^\infty_0({\bf R}_+),\quad \text{acting on } L^2({\bf R}_+). 
\end{equation}
The closure of $s_c$ is known:
\begin{proposition}\label{ananieva}(\cite[Theorem 5.1 (ii)]{Ananieva})
Let $c\in(0,1/2]$. Then, $\mathcal{D}(\overline{s_c})=H^2_0({\bf R}_+)$.
\end{proposition}
\begin{remark}\label{rmk:rangeofc}
The assumption $c\in(0,1/2]$ in Proposition \ref{ananieva} guarantees that $s_c$ is a non-negative operator. Anyway, in the recent paper \cite{besseloperator} it is proved that $\mathcal{D}(\overline{s_c})=H^2_0({\bf R}_+)$ for all $c> 0$.
\end{remark}
\begin{lemma}\label{small}
Let $c\in[0,1/2)$ and $f\in L^2(\Omega_+)$ with ${\rm supp}f\subset (0,\epsilon)\times S^1$, for some $\epsilon>0$. Then, $f\in\mathcal{D}\left(\overline{- L+cK}\right)$ if and only if $f\in\mathcal{D}(\overline{T_c})$.
\end{lemma}
\begin{proof}
Since $-L+cK=H_c+\eta_c$, and $\eta_c$ is a bounded operator on $(0,\epsilon)\times S^1$, Proposition \ref{lemma:Kato} implies that $f\in\mathcal{D}(\overline{- L+cK})$ if and only if $f\in\mathcal{D}(\overline{H_c})$, for all $c>0$. Furthermore, we want to show that the singular term $\frac{g_{1,c}}{x}:=\left(\frac{1-4c}{2}\right)\frac{\partial_x\phi}{x}$ is infinitesimally small w.r.t. $T_c$, if $c\in[0,1/2)$. In order to do this we will use two main ingredients. The first one is Hardy inequality:
$$\int_0^\infty\frac{|f(x)|^2}{x^2}dx\leq4\int_0^\infty|f'(x)|^2dx \quad \forall f\in C^\infty_0({\bf R}_+).$$
The second one is perturbation theory, in particular Kato-Rellich's Theorem (Proposition~\ref{lemma:Kato}). 

For all functions $f\in C_0^\infty({\bf R}_+\times S^1)$ we have
\begin{align*}
&  \left\|\frac{g_{1,c}f}{x} \right\|^2_{L^2({\bf R}_+\times S^1)}  =  \int_{{\bf R}_+\times S^1} \frac{|g_{1,c}(x,y)|^2}{x^2} |f(x,y)|^2 dxdy  
\\
 = & \left(1- \frac{3-8c}{4-8c} \right) \int_{{\bf R}_+\times S^1} |g_{1,c}(x,y)|^2\frac{|f(x,y)|^2}{x^2}  dxdy + \frac{3-8c}{4-8c}\int_{{\bf R}_+\times S^1}|g_{1,c}(x,y)|^2 \frac{|f(x,y)|^2}{x^2}  dxdy  
\\
\leq &\||g_{1,c}|^2\|_{L^\infty({\bf R}_+\times S^1)}\left( \frac{4}{4-8c}\int_{{\bf R}_+\times S^1}  |\partial_x f(x,y)|^2  dxdy + \frac{3-8c}{4-8c}\int_{{\bf R}_+\times S^1} \frac{|f(x,y)|^2}{x^2}  dxdy\right)
\\ 
\leq &\frac{4\||g_{1,c}|^2\|_{L^\infty({\bf R}_+\times S^1)}}{4-8c}\left( \int_{{\bf R}_+\times S^1}  |\partial_x f(x,y)|^2 +x^2e^{2\phi}|\partial_y f(x,y)|^2  + \left(\frac{3}{4} - 2c \right) \frac{|f(x,y)|^2}{x^2}  dxdy\right)
\\
=&  \frac{4\||g_{1,c}|^2\|_{L^\infty({\bf R}_+\times S^1)}}{4-8c}( T_{c}f, f )_{L^2({\bf R}_+\times S^1)} \\
\leq& \frac{2\||g_{1,c}|^2\|_{L^\infty({\bf R}_+\times S^1)}}{4-8c}\left(\delta\| T_{c}f\|^2_{L^2({\bf R}_+\times S^1)} +\frac{1}{\delta} \|f\|^2_{L^2({\bf R}_+\times S^1)}\right),
\end{align*}

where we have used Fubini's Theorem and the Hardy inequality in the first inequality, we have integrated by parts in the third equality, and we have used the Young inequality in the last inequality that holds for every $\delta>0$. This proves that $\frac{g_{1,c}}{x}$ is infinitesimally small w.r.t. $T_{c}$ if $c\in[0,1/2)$.
Proposition~\ref{lemma:Kato} then implies that 
$D(\overline{H_{c}})=D(\overline{T_{c}})$ if $c\in[0,1/2)$. \hfill$\Box$\end{proof}

\begin{remark}\label{rmk:important}
The assumption $c\in(0,1/2)$ is used in the proof of Lemma \ref{small} to guarantee the non-negativity of $T_c$.
\end{remark}

For any function $f\in L^2({\bf R}_+\times S^1)$, we denote by $f=\sum_{k\in {\bf Z}}\widehat{f}_k(x)e^{\mathrm{i}ky}$ its Fourier series. 
\begin{lemma}\label{closure}
Let $c\in(0,1/2)$, and let $f\in\mathcal{D}(\overline{- L+cK})$ be a function supported in $(0,\epsilon)\times S^1$, for some $\epsilon>0$. Then, $\widehat{f}_k(x)=o(x^\frac{3}{2})$ as $x\rightarrow 0^+$, for every $k\in {\bf Z}$.
\end{lemma}
\begin{proof}
Let $f\in C^\infty_0((0,\epsilon)\times S^1)$. Lemma~\ref{small} shows that $f\in\mathcal{D}(\overline{- L+cK})$ if and only if $f\in\mathcal{D}(\overline{T_c})$. Thus, we are left to study the behavior near $x=0$ of a function $f\in\mathcal{D}(\overline{T_c})$. For any $k\in{\bf Z}$, we have
$$\widehat{(T_c\; f)}_k=s_c\widehat{f}_k-\frac12 x^2\sum_{m+m'=k}(-m^2)\widehat{f}_m\widehat{(e^{2\phi})}_{m'}-x^2\sum_{m+m'=k}(\mathrm{i}m)\widehat{f}_m\widehat{(\partial_y \phi \,e^{2\phi})}_{m'}, $$
where $s_c$ is defined in \eqref{1D}, and we compute the norm using the triangular inequality
\begin{align*}
&\|\widehat{(T_c\; f)}_k \|_{L^2({\bf R}_+)}\geq \Big\|s_c\widehat{f}_k \Big\|_{L^2({\bf R}_+)}\\
-&\Big\|\frac12 x^2\sum_{m+m'=k}(-m^2)\widehat{f}_m\widehat{(e^{2\phi})}_{m'}\Big\|_{L^2({\bf R}_+)}-\Big\|x^2\sum_{m+m'=k}(\mathrm{i}m)\widehat{f}_m\widehat{(\partial_y \phi \,e^{2\phi})}_{m'}\Big\|_{L^2({\bf R}_+)}.
\end{align*}
We have, 
\[
\begin{split}
&\Big\|\sum_{m+m'=k}(-m^2)\widehat{f}_m\widehat{(e^{2\phi})}_{m'}\Big\|_{L^2({\bf R}_+)}\\ =&\Big\|\sum_{m+m'=k}(-m'^2)\widehat{f}_m\widehat{(e^{2\phi})}_{m'}-k^2\sum_{m+m'=k}\widehat{f}_m\widehat{(e^{2\phi})}_{m'}+2k\sum_{m+m'=k}m'\widehat{f}_m\widehat{(e^{2\phi})}_{m'}\Big\|_{L^2({\bf R}_+)} \\ \leq &
\| f\|_{L^2({\bf R}_+\times S^1)} \Big(\sup_{x\in{\bf R}_+}\|e^{2\phi} \|_{H^2(S^1)}+k^2\sup_{x\in{\bf R}_+}\|e^{2\phi} \|_{L^2(S^1)}+2|k|\sup_{x\in{\bf R}_+}\|e^{2\phi} \|_{H^1(S^1)}\;\Big) \\ \leq &C_{\phi,k}\;\| f\|_{L^2({\bf R}_+\times S^1)}
\end{split}
\]
thanks to the Cauchy-Schwartz inequality and the Plancherel formula. Similarly, 
\begin{align*}
&\Big\|\sum_{m+m'=k}(\mathrm{i}m)\widehat{f}_m\widehat{(\partial_y \phi \,e^{2\phi})}_{m'}\Big\|_{L^2({\bf R}_+)} \leq \\
&\| f\|_{L^2({\bf R}_+\times S^1)}\Big(\sup_{x\in{\bf R}_+}\|e^{2\phi}\partial_y\phi \|_{H^1(S^1)}+|k|\sup_{x\in{\bf R}_+}\|e^{2\phi} \partial_y\phi\|_{L^2(S^1)}\;\Big) \leq C'_{\phi,k}\;\| f\|_{L^2({\bf R}_+\times S^1)}
\end{align*}
Thus
\begin{equation}\label{H20}
\begin{split}
\Big\|s_c\widehat{f}_k \Big\|_{L^2({\bf R}_+)}&\leq \|\widehat{(T_c\; f)}_k \|_{L^2({\bf R}_+)}+\epsilon^2C''_{\phi,k}\| f\|_{L^2({\bf R}_+\times S^1)}\\ &
\leq \|T_c\; f \|_{L^2({\bf R}_+\times S^1)}+\epsilon^2C''_{\phi,k}\| f\|_{L^2({\bf R}_+\times S^1)}  \;\; \forall\;f\in C^\infty_0((0,\epsilon)\times S^1).
\end{split}
\end{equation}
Now, if $f\in\mathcal{D}(\overline{T}_c)$ and ${\rm supp}f\subset (0,\epsilon)\times S^1$, we can consider (thanks to a localization argument) a sequence $\{\varphi_n\}_{n\in{\bf N}}\subset C^\infty_0((0,\epsilon)\times S^1)$ that converges to $f$ in the norm of $T_c$, and hence \eqref{H20} implies that the sequence of the $k^{\rm th}$-Fourier component $\{\widehat{(\varphi_n)}_k\}_{n\in{\bf N}}\subset C^\infty_0(0,\epsilon)$ converges to $\widehat{f}_k$ in the norm of $s_c$, for all $k\in{\bf Z}$. Thus, $\widehat{f}_k\in\mathcal{D}(\overline{s_c}) \,\forall k\in{\bf Z}$. Then, the conclusion follows by applying Proposition \ref{ananieva}, since every function $f\in H^2_0({\bf R}_+)$ satisfies $f(x)=o(x^\frac{3}{2})$ for $x\rightarrow 0$.
\hfill$\Box$\end{proof}
\subsection{Adjoint operator}
\label{subsec:proof2}
We first consider the 1D Sturm-Liouville model operator given by
\begin{equation}\label{eq:modeloperator}
A=-\dfrac{d^2}{dx^2}+\dfrac{g_2}{x^2}+\dfrac{g_1}{x}, \qquad g_1,g_2\in{\bf R}. 
\end{equation}
Moreover, we introduce a $C^\infty$ cut-off function $0\leq P_\epsilon\leq1$,
\begin{equation}\label{cutoff}
P_\epsilon(x)=\begin{cases}
1 & \text{if } x\leq \epsilon/2,\\
0 & \text{if } x\geq \epsilon.
\end{cases} 
\end{equation}
\begin{lemma}\label{modeloperator} 
Let $g_1\in {\bf R}$ and $g_2\in(-1/4,3/4)$. Consider the operator $A$ acting on the Hilbert space $L^2({\bf R}_+)$ with domain $C_0^\infty({\bf R}_+)$. Then,
\begin{enumerate}
\item[(a)] for any $f\in\mathcal{D}(\overline{A})$, $f(x)=o(x^\frac{3}{2})$, as $x\rightarrow0$;
\item[(b)] $\mathcal{D}(A^*)=\mathcal{D}(\overline{A})+\mathrm{span}\{\psi_+P_\epsilon,\psi_-P_\epsilon\}$, where $P_\epsilon$ is the cut-off function defined in \eqref{cutoff}, and 
$$\psi_\pm(x)=x^{\alpha_\pm}+a_\pm\;x^{\alpha_\pm+1}, \;\; \alpha_\pm=\frac{1}{2}\pm\frac{1}{2}\sqrt{4g_2+1},\;\; a_\pm=\frac{g_1}{(\alpha_\pm+1)\alpha_\pm-g_2}, $$
if 
$g_2\neq0$,
and
$$\psi_+(x)=x, \;\; \psi_-(x)=1+ g_1 x\log(x),  $$
if $g_2=0$.
\end{enumerate}
\end{lemma}
\begin{remark}\label{rmk:generalc}
Lemma \ref{modeloperator} is already known in the literature and holds also when $g_2<3/4$ (see, e.g., \cite[Proposition 3.1]{coulomb}), with the following modification of (b) 
$$\psi_+(x)=x^{\frac{1}{2}}+g_1x^{\frac{3}{2}}, \;\; \psi_-(x)=x^{\frac{1}{2}}\log(x)+g_1x^{\frac{3}{2}}\log(x) + 2 x^{\frac{1}{2}},  $$
if $g_2=-1/4$. Below we provide an alternative proof that uses perturbative arguments and does not need an introduction of the Bessel functions. 
\end{remark}
\begin{proof}\emph{of Lemma \ref{modeloperator}}\\
To show (a), we claim that $\frac{g_1}{x}$ is infinitesimally-small w.r.t. $-\frac{d^2}{dx^2}+\frac{g_2}{x^2}$ if $g_2>-1/4$ (exactly as in  in Lemma \ref{small} we showed that $\frac{g_{1,c}}{x}$ is infinitesimally-small w.r.t. $T_c$ if $c<1/2$). Indeed, for $g_2>-1/4$ and all $f\in C^\infty_0({\bf R}_+)$ we have
\begin{align*}
  \left\|\frac{g_{1}f}{x} \right\|^2_{L^2({\bf R}_+)}  = &  \int_{{\bf R}_+} \frac{g_1^2}{x^2} |f(x)|^2 dx 
\\
 = & \left(1- \frac{4g_2}{1+4g_2} \right) \int_{{\bf R}_+} g_1^2\frac{|f(x)|^2}{x^2}  dx +\frac{4g_2}{1+4g_2}\int_{{\bf R}_+}g_1^2 \frac{|f(x)|^2}{x^2}  dx  
\\
\leq & \frac{4g_1^2}{1+4g_2}\int_{{\bf R}_+}  |f'(x)|^2  + \frac{g_2}{x^2}|f(x)|^2  dx
\\ 
=&\frac{4g_1^2}{1+4g_2}  \left(\left[-\frac{d^2}{dx^2}+\frac{g_2}{x^2}\right] f, f \right)_{L^2({\bf R}_+)} \\
\leq& \frac{4g_1^2}{1+4g_2}\left(\delta\left\|\left[-\frac{d^2}{dx^2}+\frac{g_2}{x^2}\right] f\right\|^2_{L^2({\bf R}_+)} +\frac{1}{\delta} \|f\|^2_{L^2({\bf R}_+)}\right),
\end{align*}
where we have used the Hardy inequality in the first inequality, we have integrated by parts in the third equality, and we have used the Young inequality in the last inequality that holds for every $\delta>0$. This proves the claim.
As a consequence, for $g_2\in(-1/4,3/4)$ we have
$$\mathcal{D}\left(\overline{-\dfrac{d^2}{dx^2}+\dfrac{g_2}{x^2}+\dfrac{g_1}{x}}\right)=\mathcal{D}\left(\overline{-\dfrac{d^2}{dx^2}+\dfrac{g_2}{x^2}}\right)=H^2_0({\bf R}_+), $$
where we have used Kato-Rellich's Theorem (Propositon \ref{lemma:Kato}) in the first equality and Proposition \ref{ananieva} in the second equality.

To prove the second statement, we look for the solutions of 
\begin{equation}\label{model}
-u''(x)+\dfrac{g_2}{x^2}u(x)+\dfrac{g_1}{x}u(x)=0.
\end{equation}
These are two linearly independent functions which can be expressed via confluent hypergeometric functions,
but since we are only interested in their behavior near $x=0$, we can just use the Frobenius method (see, for instance, \cite[Chapter 4]{ode2}) to understand their asymptotics. 

The first step is to write down the indicial polynomial, which is defined as
$$
P(\alpha)=(x^{-\alpha+2}A x^\alpha)|_{x=0} = \alpha(\alpha-1)-g_2.
$$
The construction depends whether or not the two roots of this polynomial are separated by an integer. The two roots are given by
$$
\alpha_\pm=\frac{1}{2}\pm\frac{1}{2}\sqrt{4g_2+1}.
$$
Under the stated range of $g_2$ it follows that the only two cases where the two roots are separated by an integer are given by $g_2 =0$.

Assume that $g_{2}\neq 0$. Then the Frobenius method states that there exist two independent solutions, which can be represented as converging series of the form
\begin{equation}
\label{frobenius}
u_{\pm}(x) = x^{\alpha_\pm}\sum_{i=0}^\infty a_i x^i.
\end{equation}

We plug the ansatz \eqref{frobenius} into \eqref{model} and obtain the following conditions for the dominating terms
\[
\begin{cases}
a_0[\alpha(\alpha-1)-g_2]=0,\\
a_1(\alpha+1)\alpha-a_1g_2+a_0g_1=0.
\end{cases}
\]
Setting $a_0=1$, we obtain that $\alpha_{\pm}$ are exactly the roots of the indicial polynomial, that
$$a_{1,\pm}=\frac{g_1}{(\alpha_\pm+1)\alpha_\pm-g_2}=:a_\pm, $$
and that the solutions are $$u_\pm(x)=x^{\alpha_\pm}+a_\pm\;x^{\alpha_\pm+1}+o(x^{\alpha_\pm+1}).$$

Assume now that $g_2 =0$. Then the Frobenius method tells us that $u_+(x)$ is still a solution of~\eqref{model} and the second solutions is given by
$$
u_-(x) = Cu_+(x)\log(x) + x^{\alpha_-}\sum_{i=0} a_i x^i.
$$
Plugging this series expression into~\eqref{model} allows us to recover $\psi_\pm$ as the dominating terms of $u_\pm$. Moreover notice that, as a direct consequence of the Frobenius method, $A(u_\pm-\psi_\pm)$ is bounded near $x=0$, and hence $A\psi_\pm\in L^2(0,1)$.
	
So, let $\psi_\pm$ as in the statement. Then
$$(i)\;\; \psi_\pm\in L^2(0,1),\;\;\; (ii)\;\;A\psi_\pm\in L^2(0,1),\;\;\;(iii)\;\;\psi_\pm \notin \mathcal{D}(\overline{A}),$$
where (i) and (ii) imply at once that $\psi_\pm P_\epsilon \in \mathcal{D}(A^*)$ and (iii) follows from part (a) and the asymptotics of $\psi_\pm$ near $x=0$. Since the functions $\psi_+P_\epsilon$ and $\psi_-P_\epsilon$ are linearly independent  and the quotient $\mathcal{D}(A^*)/\mathcal{D}(\overline{A})$ has dimension at most $2$ (as it follows from the fact that $A$ is in the limit point case at $\infty$, by applying Proposition \ref{completeinfty}, which in turns implies that $\ker(A^+\pm\ii)$ have at most dimension $1$, by applying Proposition \ref{weyl}), the thesis follows.\hfill$\Box$\end{proof}

Now, we can use Lemma \ref{modeloperator} to obtain informations on the adjoint of the 2D operator we are interested in, that is, $- L+cK$ defined in \eqref{unitarylaplacian}, and complete the proof of Theorem \ref{Grushin}.
\begin{proof}\emph{of Theorem \ref{Grushin}}\;
We take the coefficient of $\frac1x$ evaluated at $x=0$ (i.e., on the singularity) and treat the second variable $y$ as a parameter. Indeed, setting 
\begin{equation}
\label{eq:some_defs}
g_2 = \frac{3}{4}-2c, \qquad g_1(y)=\frac{1-4c}{2}\partial_x \phi(0,y)\in C^\infty(S^1)
\end{equation}
we obtain from Lemma \ref{modeloperator} two functions $\psi_{\pm,c}\in C^\infty(\Omega)$ of both variables $x,y$. 
Then, we get the following:
\begin{lemma}
Let $c\in(0,1/2)$, and define $\widetilde{h}_{\pm,\epsilon,c}(x,y)=\psi_{\pm,c}(x,y)P_\epsilon(x)\in L^2(\Omega_+)\cap C^\infty (\Omega_+)$, where $\psi_{\pm,c}$ have the same form as functions $\psi_{\pm}$ from Lemma~\ref{modeloperator} with $g_1$, $g_2$ given by~\eqref{eq:some_defs} and $P_\epsilon$ is defined in \eqref{cutoff}. Then,
\begin{enumerate}
\item [(i)] $\widetilde{h}_{\pm,\epsilon,c} \in \mathcal{D}((- L+cK)^*)$;
\item [(ii)]$\widetilde{h}_{\pm,\epsilon,c} \notin \mathcal{D}(\overline{- L+cK})$;
\item [(iii)] $\mathrm{supp}(\widetilde{h}_{\pm,\epsilon,c})\subset (0,\epsilon)\times S^1$.
\end{enumerate}
\end{lemma} 
\begin{proof}
Part (iii) is obvious, as $\mathrm{supp}(P_\epsilon)\subset(0,\epsilon)$. To prove (i), we consider the operator $R$ on the domain $C^\infty_0(\Omega_+)$, whose action is defined by $R:=(-L+cK)-A$, (where $A$ is the operator whose action is defined in \eqref{eq:modeloperator}, but now is considered on the domain $C^\infty_0(\Omega_+)$, and $g_1,g_2$ are the functions defined in \eqref{eq:some_defs}), and we claim that $R \widetilde{h}_{\pm,\epsilon,c}\in L^2(\Omega_+)$, in the weak sense. Indeed,
we have
\[
\begin{split}
R&=-x^2e^{2\phi}\partial_y^2-2x^2e^{2\phi}(\partial_y\phi)\partial_y+\dfrac{1-4c}{2}\Big(\dfrac{\partial_x\phi(x,y)-\partial_x\phi(0,y)}{x}\Big)+\eta_c\\ &
=-x^2e^{2\phi}\partial_y^2-2x^2e^{2\phi}(\partial_y\phi)\partial_y+\widetilde{\eta}_c,
\end{split}
\]
(where $\eta_c$ and $\widetilde{\eta}_c$ are bounded functions on $(0,\epsilon)\times S^1$) and the claim follows from the $C^\infty$-regularity of $\widetilde{h}_{\pm,\epsilon,c}$ w.r.t. $y$. Then, by the very construction of $\widetilde{h}_{\pm,\epsilon,c}$, we have that $A \widetilde{h}_{\pm,\epsilon,c}\in L^2(\Omega_+)$, which in turns implies that $(-L+cK)\widetilde{h}_{\pm,\epsilon,c}=(R+A)\widetilde{h}_{\pm,\epsilon,c}\in L^2(\Omega_+)$ in the weak sense, and proves part (i).

To prove part (ii) we notice that the $0^{\rm th}$-Fourier component of $\widetilde{h}_{\pm,\epsilon,c}$ is given by 
\[
\widehat{(\widetilde{h}_{\pm,\epsilon,c})}_0=P_\epsilon(x)\cdot
\begin{cases}
x^{\alpha_{\pm,c}}+\widehat{(a_{\pm,c})}_0x^{\alpha_{\pm,c}+1}&\text{if } c\neq 3/8\\
x&\text{if } c= 3/8 \text{ and }+\\
1+\widehat{(g_1)}_0 x\log(x)&\text{if } c= 3/8 \text{ and }-
\end{cases}
\]
where $\alpha_{\pm,c}=\frac{1}{2}\pm\sqrt{1-2c}$ and $a_{\pm,c}$ is the function found in Lemma \ref{modeloperator} (b) w.r.t. the function $g_1$ and the constant $g_2$ defined in \eqref{eq:some_defs}. The conclusion follows by applying Lemma \ref{closure}, since $\widehat{(\widetilde{h}_{\pm,\epsilon,c})}_0$ is not $o(x^\frac{3}{2})$ for any $c\in(0,1/2)$.
\hfill$\Box$\end{proof}
To conclude the proof of Theorem \ref{Grushin}, it suffices to consider $h_{\pm,\epsilon,c}:=U^{-1}\widetilde{h}_{\pm,\epsilon,c}$, where $U$ is the unitary transformation defined in \eqref{U}.
\hfill$\Box$\end{proof}
The proof of Theorem \ref{infinitedeficency} is now an immediate consequence:
\begin{proof}\emph{of Theorem \ref{infinitedeficency}}\;
It follows by considering the infinite-dimensional vector space spanned by the family of functions $\{h_{\pm,\epsilon,c}e^{\mathrm{i}ky}\}_{k\in{\bf Z}}\subset \mathcal{D}((-\Delta+cK)^*)\setminus\mathcal{D}(\overline{-\Delta+cK}) $.\hfill$\Box$\end{proof}

\section{Proof of Theorem \ref{theorem}}
\label{sec:proofonmanifold}
If $c=0$, $H_0=-\Delta$ is known to be essentially self-adjoint on $L^2(M,\omega)$ (\cite[Theorem 1.1]{Boscain-Laurent-2013}). Then, let $c\in(0,1/2)$. \\

Let $Z=\coprod_{j\in J}W_j$ be the disjoint union in connected components for the singular set, and $M=\cup_{i\in I}\Omega_i$ be an open cover such that, for every $W_j$, there exist a unique $\Omega_{i_j}$ (Grushin zone) with $W_j \subset \Omega_{i_j}$ and $W_j\cap \Omega_{i} = \emptyset$ if $i\neq i_j$. Moreover, as previously remarked, we can assume that $\Omega_{i_j}$ is a tubular neighborhood of $W_j$, i.e., $\Omega_{i_j} \cong {\bf R}\times S^1$.\\

Let $W$ be a connected component of $Z$, and $\Omega$ the corresponding Grushin zone. Consider the operator $(-\Delta+cK)_\Omega$ defined as the restriction of $-\Delta+cK$ on the domain $C^\infty_0(\Omega\setminus W)$. In the local chart $\Omega$ with coordinates $(x,y)\in {\bf R}\times S^1$, $W=\{(x,y)\in {\bf R}\times S^1 \mid x=0\}$, and Theorem \ref{Grushin} gives a function, e.g. $h_{+,\epsilon,c}$, supported arbitrarily close to $W$, such that $h_{+,\epsilon,c}\in \mathcal{D}((-\Delta+cK)_\Omega^*)\setminus \mathcal{D}(\overline{(-\Delta+cK)}_\Omega)$.\\

We define the function 
\[
F_{\epsilon,c}=
\begin{cases}
h_{+,\epsilon,c} & \text{on $\Omega$}, \\
0 & \text{on $M\setminus \Omega$}.
\end{cases}
\]
So we have
\[
\begin{split}
(F_{\epsilon,c},(-\Delta+cK)u)_{L^2(M)}&=(F_{\epsilon,c},(-\Delta+cK)u)_{L^2(\Omega)}+(F_{\epsilon,c},(-\Delta+cK)u)_{L^2(M\setminus \Omega)}\\ &
=((-\Delta+cK)_\Omega^*h_{+,\epsilon,c},u)_{L^2(\Omega)},\quad \forall u\in C^\infty_0(M\setminus Z), 
\end{split}
\]
having integrated by parts ($h_{+,\epsilon,c}$ vanishes away from $\partial \Omega$, and $u$ vanishes away from $W$), which proves that
\[
(-\frac12\Delta+cK)^*F_{\epsilon,c}=
\begin{cases}
(-\Delta+cK)^*_\Omega h_{+,\epsilon,c} & \text{on $\Omega$}, \\
0 & \text{on $M\setminus \Omega$},
\end{cases}
\]
and $F_{\epsilon,c}\in \mathcal{D}((-\Delta+cK)^*)$. We are left to prove that $F_{\epsilon,c}\notin \mathcal{D}(\overline{-\Delta+cK})$, which implies the non-self-adjointness of $-\Delta+cK$ on $L^2(M)$. \\

Suppose by contradiction that $F_{\epsilon,c}\in \mathcal{D}(\overline{-\Delta+cK})$. Then, there exist a sequence $(\phi_n)_{n\in\mathbb{N}}\subset C^\infty_0(M\setminus Z)$ and a function $G_{\epsilon,c}\in L^2(M)$ such that
\begin{enumerate}
\item [(i)] $\phi_n\rightarrow F_{\epsilon,c}$ , as $n\rightarrow \infty$, in $L^2(M)$,
\item [(ii)] $(-\Delta+cK)\phi_n\rightarrow G_{\epsilon,c}$ , as $n\rightarrow \infty$, in $L^2(M)$. 
\end{enumerate}
Now, $G_{\epsilon,c}$ must satisfy
$$G_{\epsilon,c}=(\overline{-\Delta+cK})F_{\epsilon,c}=(-\Delta+cK)^*F_{\epsilon,c}=
\begin{cases}
(-\Delta+cK)^*_\Omega h_{+,\epsilon,c} & \text{on $\Omega$}, \\
0 & \text{on $M\setminus \Omega$},
\end{cases}. $$
So, $F_{\epsilon,c}$ and $G_{\epsilon,c}$ are both supported in $U\subsetneq \Omega$. We then consider the cut-off function $\xi\in C^\infty_0(\Omega)$
\begin{equation}
\xi(x)=\begin{cases}
1 & \text{if }x \in U,\\
0 & \text{if } x \notin \Omega,
\end{cases} 
\end{equation}
with $0\leq\xi\leq 1$, and define the sequence $(\widetilde{\phi}_n=\xi\phi_n)_{n\in\mathbb{N}}\subset C^\infty_0(\Omega\setminus W)$. We have the following
\begin{lemma}\label{contradiction}
$\widetilde{\phi}_n\rightarrow h_{+,\epsilon,c}$ and $(-\Delta+cK)\widetilde{\phi}_n=(-\Delta+cK)_\Omega\widetilde{\phi}_n\rightarrow G_{\epsilon,c}|_{\Omega}$ , as $n\rightarrow \infty$, in $L^2(\Omega)$. 
\end{lemma}
Thus, we conclude by applying Lemma \ref{contradiction} which says that  $h_{+,\epsilon,c}\in \mathcal{D}((\overline{-\Delta+cK)}_\Omega)$, which is impossible.\\
\begin{proof}\emph{of Lemma \ref{contradiction}}\;\;
Because of (i) and (ii), we have as $n\rightarrow \infty$
\[
\begin{split}
(i.1)\;\; \|\phi_n-h_{+,\epsilon,c}\|_{L^2(U)}&\rightarrow 0\;, \quad (i.2)\;\; \|\phi_n \|_{L^2(M\setminus U)}\rightarrow 0, \\
(ii.1)\;\;\|(-\Delta+cK)\phi_n-G_{\epsilon,c}\|_{L^2(U)}&\rightarrow 0\;, \quad (ii.2)\;\; \|(-\Delta+cK)\phi_n \|_{L^2(M\setminus U)}\rightarrow 0,
\end{split}
\]
since $\mathrm{supp}(h_{+,\epsilon,c})$ and $\mathrm{supp}(G_{\epsilon,c})$ are both contained in $U$. Then we have (as $n \to \infty$)
$$ \|\widetilde{\phi}_n-h_{+,\epsilon,c}\|_{L^2(\Omega)}= \|\phi_n-h_{+,\epsilon,c}\|_{L^2(U)}+\|\xi \phi_n \|_{L^2(\Omega\setminus U)}\leq  \|\phi_n-h_{+,\epsilon,c}\|_{L^2(U)}+\|\phi_n \|_{L^2(M\setminus U)}\rightarrow 0.$$
Moreover, using that $\Delta(\xi \phi_n)=(\Delta \xi)\phi_n+2\nabla \xi \cdot \nabla \phi_n +\xi(\Delta\phi_n)$, we have 
\[
\begin{split}
\|(-\Delta+cK)\widetilde{\phi}_n&-G_{\epsilon,c}|_{\Omega}\|_{L^2(\Omega)} \leq \\ & \|(-\Delta+cK)\phi_n-G_{\epsilon,c}\|_{L^2(U)}+C\|(-\Delta+cK)\phi_n+|\nabla \phi_n|+\phi_n \|_{L^2(\Omega \setminus U)},
\end{split}
\]
where $C$ is a constant such that $C>\|\Delta\xi \|_{L^\infty(\Omega)}, \|\;|\nabla\xi|\; \|_{L^\infty(\Omega)},  \|\xi \|_{L^\infty(\Omega)} $. Since $K$ is a bounded function on $\Omega\setminus U$, we have
$$\|K \phi_n \|_{L^2(\Omega\setminus U)}\leq \|K \|_{L^\infty(\Omega\setminus U)}\cdot \| \phi_n \|_{L^2(\Omega\setminus U)} \rightarrow 0, \quad \text{as } n \to \infty,$$
and
$$\|\Delta \phi_n \|_{L^2(\Omega\setminus U)}\leq  \|(-\Delta+cK)\phi_n\|_{L^2(\Omega\setminus U)}+\|cK\phi_n \|_{L^2(\Omega \setminus U)}\rightarrow 0, \quad \text{as } n \to \infty.$$
Finally, by Sobolev embedding, we have
$$\|\;|\nabla \phi_{n}|\; \|_{L^2(\Omega\setminus U)}\leq \widetilde{C}(\|\Delta \phi_n \|_{L^2(\Omega\setminus U)}+\|\phi_n \|_{L^2(\Omega\setminus U)})\rightarrow 0, \quad \text{as } n \to \infty.$$
\hfill$\Box$\end{proof}
To prove that the deficiency indices of $-\Delta+cK$ are infinite if $c>0$, it suffices to consider the infinite-dimensional vector space spanned by the family of functions $\{F^k_{\epsilon,c}\}_{k\in{\bf Z}}$ contained in $\mathcal{D}((-\Delta+~cK)^*)~\setminus \mathcal{D}(\overline{-\Delta+~cK})$ defined by \[
F^k_{\epsilon,c}=
\begin{cases}
e^{{\rm i}ky}h_{+,\epsilon,c} & \text{on $\Omega$}, \\
0 & \text{on $M\setminus \Omega$}.
\end{cases}
\]
\begin{remark} One can construct such family of functions close to any singular region of $M$, and each singular region has an infinite family of self-adjoint extensions; this gives room to self-adjoint extensions on the whole manifold, characterized by different boundary conditions to be imposed at each singular region.
\end{remark}
This concludes the proof of Theorem \ref{theorem}.

\textbf{Acknowledgements:}
We thank M. Gallone and A. Michelangeli for very instructive discussions, and Stephen A. Fulling for sharing with us \cite{fulling}. We are also deeply grateful to the anonymous reviewer for the crucial comments and corrections. The project leading to this publication has received funding from the European Union’s Horizon 2020 research and innovation programme under the Marie Sklodowska-Curie grant agreement no. 765267 (QuSCo). Also it was supported by the ANR projects SRGI ANR-15-CE40-0018 and Quaco ANR-17-CE40-0007-01.

\bibliographystyle{siamplain}
\bibliography{references}

\end{document}